\documentclass[11pt,english]{amsart}
\def\margin_comment#1{\marginpar{\sffamily{\tiny #1\par}\normalfont}}
\usepackage{physics}
\usepackage{mathtools}
\usepackage{mathrsfs}
\usepackage{graphicx}
\usepackage{graphics}
\usepackage{setspace}
	\usepackage{color}
\usepackage{amsmath,amssymb,amsthm}
\usepackage{epsfig}
\usepackage{dutchcal}
\usepackage{babel}
\usepackage{tikz}
\usepackage{tikz-cd}
\usetikzlibrary{arrows,shapes,automata,backgrounds,petri,positioning,scopes,matrix, calc, plothandlers,plotmarks,patterns}
\usepackage{enumerate}
\usepackage{float}
\usepackage[utf8]{inputenc}
\usepackage[T1]{fontenc}
\usepackage[document]{ragged2e}
 \usepackage{verbatim}
\usepackage{tikz-3dplot}
\usepackage{pgfplots}
\usepackage{tkz-graph}
\GraphInit[vstyle = Normal]
\tikzset{
  LabelStyle/.style = {minimum width = 2em, 
                        text = red, font = \bfseries },
  VertexStyle/.append style = { inner sep=2pt,
                                font = \Large\bfseries, fill},
  EdgeStyle/.append style = {->, bend left} }

\newtheorem{thm}{Theorem}[section]
\numberwithin{equation}{section} 
\numberwithin{figure}{thm} 
\theoremstyle{plain}
\newtheorem*{thm*}{Theorem}
\theoremstyle{definition}
\theoremstyle{plain}
\newtheorem{thm_A}{Theorem}
\newtheorem*{defn*}{Definition}

\theoremstyle{plain}

\theoremstyle{plain} 

\theoremstyle{plain}
\newtheorem{prop}[thm]{Proposition} 
\theoremstyle{definition}
\newtheorem{ex}[thm]{Example}
\theoremstyle{remark}
\newtheorem{rem}[thm]{Remark}
\theoremstyle{plain}

\theoremstyle{plain}
\newtheorem{cor}[thm]{Corollary}
\theoremstyle{plain}
\newtheorem{lem}[thm]{Lemma}
\newtheorem*{lem*}{Lemma} 
\theoremstyle{definition}
\newtheorem{defn}[thm]{Definition}

\newtheorem*{acknowledgment*}{Addentum}

\theoremstyle{plain}
\newtheorem*{ex*}{Example}
\theoremstyle{plain}

\AtBeginDocument{
  
}
\begin{document}
\pgfdeclarelayer{background}
\pgfsetlayers{background,main}

\title[HS conjecture and convex polygons]{Herzog-Sch\"onheim conjecture,  vanishing sums of roots of unity and convex polygons}

\author{Fabienne Chouraqui}

\date{}

\maketitle
\begin{abstract}
Let $G$ be a group and $H_1$,\ldots,$H_s$ be subgroups of $G$ of  indices $d_1,\ldots,d_s$ respectively. In 1974, M. Herzog and J. Sch\"onheim conjectured that if $\{H_i\alpha_i\}_{i=1}^{i=s}$,  $\alpha_i\in G$, is a coset partition of $G$, then $d_1,\ldots,d_s$ cannot be distinct. In this paper, we  present the conjecture as a problem on vanishing sum of roots of unity and convex polygons and prove some results using this approach.
\end{abstract}
\maketitle
\section{Introduction}

Let $G$ be a group, $s$ a natural number,  and $H_1$,...,$H_s$ be subgroups of $G$.  If there exist  $\alpha_i\in G$ such that $G= \bigcup\limits_{i=1}^{i=s}H_i\alpha_i$, and the sets  $H_i\alpha_i$, $1 \leq i \leq s$,  are pairwise disjoint, then  $\{H_i\alpha_i\}_{i=1}^{i=s}$ is \emph{a coset partition of $G$}  (or a \emph{disjoint cover of $G$}). In this case,    all the subgroups  $H_1$,...,$H_s$ can be assumed to be of  finite index in  $G$ \cite{newman,korec}. We denote by $d_1$,...,$d_s$ the indices of $H_1$,...,$H_s$ respectively. The coset partition $\{H_i\alpha_i\}_{i=1}^{i=s}$ has  \emph{multiplicity} if $d_i=d_j$ for some $i \neq j$. The  Herzog-Sch\"onheim conjecture is true for the group $G$, if any coset partition of $G$ has multiplicity.

\setlength\parindent{10pt}	If $G$ is the infinite cyclic group $\mathbb{Z}$, a coset partition of $\mathbb{Z}$ is 
$\{d_i\mathbb{Z} +r_i\}_{i=1}^{i=s}$, $r_i \in \mathbb{Z}$,  with  
each $d_i\mathbb{Z} +r_i$ the residue class of $r_i$ modulo $d_i$.
These coset partitions of $\mathbb{Z}$ were first introduced by P. Erd\H{o}s \cite{erdos1} and he conjectured that if $\{d_i\mathbb{Z} +r_i\}_{i=1}^{i=s}$, $r_i \in \mathbb{Z}$, is a  coset partition of $\mathbb{Z}$, then  the largest index $d_s$ appears at least twice.  Erd\H{o}s' conjecture was 
proved independently by H. Davenport and R.Rado, and independently by  L. Mirsky and  D. Newman  using analysis of complex function \cite{erdos2,newman,znam}. Furthermore, it was proved that  the largest index $d_s$ appears at least $p$ times,  where $p$ is the smallest prime dividing $d_s$ \cite{newman,znam,sun2}, that each index $d_i$  divides another index $d_j$, $j\neq i$, and  that each index $d_k$ that does not properly divide any other index  appears at least twice \cite{znam}. We refer also to \cite{por1,por2,por3,por4,sun3} for more details on  coset partitions of $\mathbb{Z}$ (also called covers of $\mathbb{Z}$ by arithmetic progressions) and to \cite{ginosar} for a proof of the Erd\H{o}s' conjecture using group representations. 

In 1974, M. Herzog and J. Sch\"onheim extended Erd\H{o}s' conjecture for arbitrary groups and conjectured that if $\{H_i\alpha_i\}_{i=1}^{i=s}$,  $\alpha_i\in G$, is a  coset partition of $G$, then $d_1$,..,$d_s$ cannot be distinct. In the 1980's, in a series of papers,  M.A. Berger, A. Felzenbaum and A.S. Fraenkel studied  the Herzog-Sch\"onheim conjecture \cite{berger1, berger2,berger3} and in \cite{berger4} they proved the conjecture is true for the pyramidal groups, a subclass of the finite solvable groups. Coset partitions of finite groups with additional assumptions on the subgroups of the partition have been extensively studied. We refer to \cite{brodie,tomkinson1, tomkinson2,sun}. In \cite{schnabel}, the authors very recently proved that the conjecture is true for all groups of order less than $1440$. 

	The common approach to the Herzog-Sch\"onheim conjecture is to study it in finite groups. Indeed, given any group $G$, every coset partition of $G$ induces a coset partition of a particular finite  quotient group of $G$ with the same indices (the quotient of $G$  by the intersection of the normal cores of the subgroups from the partition)  \cite{korec}. In  \cite{chou_hs,chou-space,chou-davenport}, we adopt a completely  different approach to the   Herzog-Sch\"onheim conjecture and  in this paper,  we develop and deepen it further. Instead of  finite groups,  we consider  free groups of finite rank and we develop new tools to the problem that permit us to give conditions on the coset partition of the free group that ensure it has multiplicity. This approach has the  advantage that  it permits to  obtain results on  the   Herzog-Sch\"onheim conjecture in both the  free groups of finite rank and all the finitely generated  groups.  Indeed,  any coset partition of a finitely generated  group  $G$ induces a  coset partition of a free group with the same indices \cite{chou_hs}. \\

In order  to  study the   Herzog-Sch\"onheim conjecture in  free groups of finite rank, we  use the machinery of covering spaces. The  fundamental group of the  bouquet with $n$ leaves (or the wedge sum of $n$ circles),  $X$,  is $F_n$, the  free group of finite rank $n$ freely generated by $A=\{a_1,...,a_n\}$. As  $X$ is a ``good'' space (connected, locally path connected and semilocally $1$-connected), $X$ has a  universal covering space which can be identified with the Cayley graph of $F_n$,  an infinite simplicial tree. Furthermore, there exists a  one-to-one correspondence between the subgroups of $F_n$ and the covering spaces (together with a chosen point) of $X$.

  For any  subgroup $H$ of $F_n$ of finite index $d$, there exists  a $d$-sheeted covering space  $(\tilde{X}_H,p)$  with a fixed basepoint. The underlying graph of $(\tilde{X}_H,p)$ is a directed labelled graph with $d$ vertices. We call it  \emph{the  Schreier graph of $H$} and denote it by $\tilde{X}_{H}$. It can be seen also as  a finite complete bi-deterministic automaton; fixing the start and the end state at the basepoint, it recognises the set of elements in $H$.  It is   called \emph{the Schreier coset diagram  for $F_n$ relative to the subgroup  $H$} \cite[p.107]{stilwell} or  \emph{the Schreier automaton for $F_n$ relative to the subgroup $H$} \cite[p.102]{sims}. The $d$ vertices (or states) correspond to the $d$ right cosets of $H$,  each edge (or transition) $Hg \xrightarrow{a}Hga$, $g \in F_n$, $a$ a generator of $F_n$,  describes the right action of $a$ on  $Hg$. If we fix the start state at  $H$, the basepoint,  and the end state at another vertex  $Hg$, where $g$  denotes the label of some path from the start state to the end state, then this automaton recognises the set of elements in $Hg$ and we call it  \emph{the  Schreier automaton  of $Hg$} and denote it by $\tilde{X}_{Hg}$.\\

In general, for any  automaton $M$, with alphabet $\Sigma$, and $d$ states (or a directed graph with $d$ vertices),  there exists a  square matrix $A$ of order $d\times d$, with $a_{ij}$ equal to the number of directed edges from  vertex $i$ to vertex $j$, $1\leq i,j\leq d$. This matrix is   non-negative and it is  called  \emph{the transition matrix} \cite{epstein-zwik}. If for every $1\leq i,j\leq d$, there exists $m_{ij} \in \mathbb{Z}^+$ such that $(A^{m_{ij}})_{ij}>0$, the matrix is \emph{irreducible}. For an irreducible non-negative matrix  $A$,  \emph{the period of $A$} is  the gcd of all $m \in \mathbb{Z}^+$ such that $(A^m)_{ii} >0$ (for any $i$).  If  $i$ and $j$ denote respectively the start and end  states of $M$, then the number of  words of length $k$ (in the alphabet $\Sigma$) accepted by $M$ is $a_k=(A^k)_{ij}$. \emph{The generating function of $M$} is defined by  $p(z)=\sum\limits_{k=0}^{k=\infty}a_k\,z^k$.  It is a rational function: the fraction of two polynomials in $z$ with integer coefficients \cite{epstein-zwik}, \cite[p.575]{stanley}.\\
 
 In \cite{chou-davenport}, we initiate the study of  the transition matrices and  generating functions of the Schreier automata in the context of coset partitions of the free group.  Let $F_n=\langle \Sigma\rangle$, and $\Sigma^*$ the free monoid  generated by $\Sigma$. Let $\{H_i\alpha_i\}_{i=1}^{i=s}$ be a coset  partition of $F_n$  with $H_i<F_n$ of index $d_i>1$, $\alpha_i \in F_n$, $1 \leq i \leq s$. Let $\tilde{X}_{i}$ denote the  Schreier  graph  of $H_i$, with transition matrix $A_i$ of period $h_i\geq 1$ and $\tilde{X}_{H_i\alpha_i}$  the  Schreier  automaton of $H_i\alpha_i$, with generating function  $p_i(z)$, $1 \leq i\leq s$. For each $\tilde{X}_{i}$,  $A_i$ is a non-negative irreducible matrix and $a_{i,k}$, $k \geq 0$, counts the  number of  words of length $k$ that belong to $H_i\alpha_i\cap \Sigma^*$.
 Since $F_n$ is the disjoint union of the sets  $\{H_i\alpha_i\}_{i=1}^{i=s}$, each element  in $\Sigma^*$ belongs to one and exactly one such set, so $n^k$, the number of  words of length $k$ in $\Sigma^*$, satisfies $n^k=\sum\limits_{i=1}^{i=s}a_{i,k}$, for every $k \geq 0$, and moreover $\sum\limits_{k=0}^{k=\infty}n^k\,z^k=\sum\limits_{i=1}^{i=s}p_i(z)$.  \\
 
 Using this kind of counting argument, we prove that if $h=max\{h_i \mid 1 \leq i \leq s\}$ is greater than $1$, then there is a repetition of the maximal  period $h>1$ and   that,  under certain conditions, the coset partition has multiplicity  \cite{chou-davenport}. Furthermore, we recover the  Davenport-Rado result (or Mirsky-Newman result) for the Erd\H{o}s' conjecture and some of its consequences. Indeed, for every index $d$, the Schreier graph of $d\mathbb{Z}$ has a transition matrix with period equal to $d$, so a repetition of the period is equivalent to a repetition of the index. For the  unique subgroup $H$ of $\mathbb{Z}$ of  index $d$,  its Schreier graph  $\tilde{X}_{H}$ is a closed directed path of length $d$ (with each edge labelled $1$). So, its  transition matrix   $A$ is the permutation matrix corresponding to the $d-$cycle $(1,2,...,d)$, and it has   period $d$.  In particular, the period of $A_s$ is $d_s$, and  there exists $j\neq s$ such that $d_j=d_s$.  Also, if the period (index) $d_k$ of $A_k$ does not properly divide any other period (index),  then  there exists  $j\neq k$ such that $d_j=d_k$. \\

  In this paper, we deepen further our study of  the transition matrices and  generating functions of the Schreier automata in the context of coset partitions of the free group. Indeed,  using elements from the Perron-Frobenius theory of irreducible non-negative matrices, we study the behaviour of the generating functions at some special poles.  
  
     For every $1 \leq i \leq s$, the  transition matrix $A_i$ of  $\tilde{X}_{i}$,  the  Schreier  graph of $H_i<F_n$, is a non-negative and irreducible matrix with  Perron-Frobenius eigenvalue $n$ (the number of free generators of $F_n$).  Indeed,   as the sum of each row and each column in $A_i$  is equal to $n$, $n$ is the positive  simple eigenvalue of maximal absolute value of $A_i$. If $A$ is an irreducible non-negative matrix with period  $h>1$, then  $A$ has exactly $h$  complex simple eigenvalues of maximal absolute value:  $n\,\omega^k$, $\,\,0 \leq k \leq h-1$, where    $\omega=e^{\frac{2\pi i}{h}}$ is the root of unity of order $h$. Moreover, the matrix $A$ is similar to $\omega A$, that is  the spectrum of $A$ is invariant under multiplication by $\omega$. If the matrix $A$ is an irreducible and aperiodic non-negative matrix, then it satisfies many properties similar  to those of the positive matrices. \\
      
    Given a coset  partition of $F_n$, $P=\{H_i\alpha_i\}_{i=1}^{i=s}$, with $H_i<F_n$ of index $d_i>1$, $\alpha_i \in F_n$, and Schreier graph $\tilde{X}_{i}$ with transition matrix $A_i$, and generating function $p_i(z)$,  $1 \leq i\leq s$, we  consider the case of coset partitions with  at least one of the matrices $A_i$ not aperiodic, that is $h=max\{h_i \mid 1 \leq i \leq s\}>1$.  In this case, as said above, $\mid J \mid>1$, where $J=\{j \,\mid\, 1 \leq j\leq s,\,h_j=h\}$ (\cite{chou-davenport}) . 
    For every $j \in J$,    $\{\frac{1}{n}\,\omega^k \mid 0\leq k \leq h-1\}$, is a set of simple poles of $p_j(z)$
    and $\sum\limits_{j\in J}Res(p_j(z), \frac{1}{n}\omega)=0$, since $\sum\limits_{i=1}^{i=s}p_i(z)\,=
\,\sum\limits_{k=0}^{k=\infty}n^k\,z^k\,=\,\frac{1}{1-nz}$ and $Res(\frac{1}{1-nz}, \frac{1}{n}\omega)=0$, where $Res(f(z),z_0)$ denoted the residue of $f(z)$ at $z_0$.  Using our computations  of the residues at these simple poles, the  equation  $\sum\limits_{j\in J}Res(p_j(z), \frac{1}{n}\omega)=0$ is a vanishing sum of roots of unity of order $h$ with positive integer coefficients . Furthermore, we show that the coset  partition $P$ induces  a set of  irreducible vanishing sums of roots of unity of order $h$ with positive integer coefficients  and to each such sum there is an associated convex polygon \cite{mann}. Indeed, we show:

  \begin{thm_A}\label{theo1}
	Let $F_n$ be the free group on $n \geq 2$ generators. Let $P=\{H_i\alpha_i\}_{i=1}^{i=s}$ be a coset  partition with $H_i<F_n$ of index $d_i$, $\alpha_i \in F_n$, $1 \leq i \leq s$, and $1<d_1 \leq ...\leq d_s$.  Let $\tilde{X}_{i}$ denote the  Schreier  graph of $H_i$, with  transition matrix   $A_i$ of period $h_i\geq 1$, $1 \leq i\leq s$.  Assume $h>1$, where  $h=max\{h_i \mid 1 \leq i \leq s\}$.  
	Let $J=\{j \,\mid\, 1 \leq j\leq s,\,h_j=h\}$. Then, 	$P$ induces a    vanishing sum of roots of unity of order $h$:
	\[\sum\limits_{j\in J}(\prod\limits_{\substack{i\in J \\ i\neq j}}d_i)\,(\omega^{})^{m_{j}}\,=\,0\] 
	Furthermore,    $P$  induces a    set of irreducible vanishing sums of roots of unity of order $h$ of the following form, where $J' \subseteq J$:
	\[\sum\limits_{j\in J'}(\prod\limits_{\substack{i\in J '\\ i\neq j}}d_i)\,(\omega^{})^{m_{j}}\,=\,0\] 
	To each irreducible sum, there is associated a convex polygon $\mathcal{P}$ with $\mid J'\mid$ sides.
\end{thm_A}
Using the construction from  Theorem \ref{theo1}, we can translate the HS conjecture as a problem in terms of  planar geometry.
\begin{thm_A}\label{theo2}
	Let $F_n$ be the free group on $n \geq 2$ generators. Let $P=\{H_i\alpha_i\}_{i=1}^{i=s}$ be a coset  partition with $H_i<F_n$ of index $d_i$, $\alpha_i \in F_n$, $1 \leq i \leq s$, and $1<d_1 \leq ...\leq d_s$.  Let $\tilde{X}_{i}$ denote the  Schreier  graph of $H_i$, with  transition matrix   $A_i$ of period $h_i\geq 1$, $1 \leq i\leq s$.  Assume $h>1$, where  $h=max\{h_i \mid 1 \leq i \leq s\}$.  	Then $P$ has multiplicity if and only if  there is an associated  convex polygon $\mathcal{P}$  with at least two edges of the same length. 
\end{thm_A}

Using the construction from Theorem \ref{theo1}, and  general results on  irreducible vanishing sum of roots of unity, we prove, under certain conditions on  $h$, that  a coset partition has multiplicity.

\begin{thm_A}\label{theo3}
	Let $F_n$ be the free group on $n \geq 2$ generators. Let $P=\{H_i\alpha_i\}_{i=1}^{i=s}$ be a coset  partition with $H_i<F_n$ of index $d_i$, $\alpha_i \in F_n$, $1 \leq i \leq s$, and $1<d_1 \leq ...\leq d_s$.  Let $\tilde{X}_{i}$ denote the  Schreier  graph of $H_i$, with  transition matrix   $A_i$ of period $h_i\geq 1$, $1 \leq i\leq s$.  Assume $h>1$, where  $h=max\{h_i \mid 1 \leq i \leq s\}$.
	Then $P$ has multiplicity  in the following cases:
	\begin{enumerate}[(i)]
		\item If $h=p^n$,  where $p$ is a  prime.
		\item If $h=p^nq^m$, where $p,q$ are primes. 
	\end{enumerate}
Furthermore, in these cases, all the associated  convex polygons  are regular.
\end{thm_A}

    The paper is organized as follows.  In  Section $2$, we give some preliminaries on vanishing sums of roots of unity, automata and their generating functions, and non-negative irreducible matrices. In Section $3$, we present some preliminary results on the Schreier automaton of a coset of a subgroup of $F_n$, and on  its  generating function. We prove some properties of  the generating function and in particular we compute the residues  of the generating function  at some special poles. In Section $4$,   we present the irreducible vanishing sum of roots of unity and its associated polygon induced by a coset partition, and  we  prove the main results.    In many places, we write the HS conjecture instead of  the Herzog-Sch\"onheim conjecture.

 \section{Preliminaries}
 \subsection{About vanishing sum of roots of unity and their induced convex polygons}
 We refer the reader to \cite{schonberg}, \cite{mann}, \cite{lenstra}, \cite{lam}.  A \emph{vanishing sum of roots of unity of length $k$} is  an equation $(*)$ of the form 
 $\sum\limits_{j=1}^{j=k}a_j\zeta_j\,=\,0$, where the $a_j$ belong to $\mathbb{C}^*$ and the $\zeta_j$ are roots of unity.  
 In \cite{schonberg} and \cite{mann}, the authors consider such equations with $a_j$ in  $\mathbb{Z}$ and we use their terminology. If the coefficients $a_j$ are positive integers, then $(*)$ can be interpreted as  a convex $k$-sided polygon with integral sides whose angles are rational when measured in degrees. In this case, $(*)$ is also called a \emph{$k$-sided polygon}. The  equation $(*)$ is called \emph{degenerate} if two of the $\zeta_j$ are equal. It is called \emph{irreducible} if there is no relation  $\sum\limits_{j=1}^{j=k}b_j\zeta_j\,=\,0$, $b_j(a_j-b_j)=0$, $1 \leq j \leq k$, where at least one but not all $b_j=0$, that is there is no proper non-empty subsum that vanishes.  It is called \emph{primitive} if $gcd(a_1,..,a_k)=1$ and if there is no non-empty  relation  $\sum\limits_{j=1}^{j=k}b_j\zeta_j\,=\,0$, where  $b_j=0$ for at least one $j$, that is  $k$ is the minimal length of  a vanishing sum with the $\zeta_j$.  It is called \emph{minimal} if there is no non-empty  relation  $\sum\limits_{j=1}^{j=k}b_j\zeta_j\,=\,0$, where $0 \leq b_j\leq a_j$ for every $j$ . Every polygon is a linear combination with positive  coefficients of minimal polygons. A primitive  equation $\sum\limits_{j=1}^{j=k}a_j\zeta_j\,=\,0$ with positive coefficients is \emph{minimal} \cite[Th.3]{mann}. Every $k$-sided polygon may be obtained from a finite set of minimal polygons of $k$ or fewer sides and there is only a finite number of classes of congruent minimal polygons of given side \cite{mann}.
 \begin{thm}\cite[Th.1]{mann}\label{thm-mann}
If  $(*)$ $\sum\limits_{j=1}^{j=k}a_j\zeta_j\,=\,0$, with $a_j \in \mathbb{Z}$,  is irreducible, then there are distinct primes $p_1,p_2,...,p_t$ where $p_1<p_2<...<p_t\leq k$ and $p_1p_2...p_t$-th roots of unity $\eta_j$ such that $\zeta_j=\eta_j\xi$, $1 \leq j \leq k$, $\xi$ any root of unity. Moreover, if $(*)$ is an  irreducible polygon and if we cannot choose $p_t<k$, then we can choose $t=1$. In the latter case all $a_j$ are equal and $(*)$ represents a regular $k$-sided polygon.
 \end{thm}
 From Theorem \ref{thm-mann},  any irreducible  equation $\sum\limits_{j=1}^{j=k}a_j\zeta_j\,=\,0$ can be transformed into an irreducible equation $\sum\limits_{j=1}^{j=k}a_j\omega^{m_j}\,=\,0$, where $\omega=e^{\frac{2\pi i}{h}}$ is a root of unity of order $h$ with $h$ a divisor of  $p_1p_2...p_t$ (up to rotation by some $\xi$). 
 In \cite{mann}, there is a classification of the primitive and irreducible $k$-sided polygons with $k \leq 7$. Most of them are regular and in the other cases they have at least two edges of the same length.\\

 In \cite{lam}, the authors study the following question:  given a natural number $h$, what are the possible values of the length $k$ of a vanishing sum of roots of unity of order $h$ ? They show that for any $h=p_1^{n_1}p_2^{n_2}...p_t^{n_t}$, where $p_1,..,p_t$ are different primes, $k$ belongs to $\mathbb{N}p_1+...+\mathbb{N}p_t$ \cite[p.92]{lam}. This result implies that any non-empty vanishing sum of $h$-th roots of unity must have length at least $p_1$, where $p_1$ is the smallest prime dividing $h$. Furthermore, they show:
 \begin{thm}\label{lam}\cite[Thm2.2, Cor.3.4]{lam}
 Let  $\sum\limits_{j=1}^{j=k}a_j\omega^{m_j}\,=\,0$ be a vanishing sum of roots of unity of order $h$, $\omega=e^{\frac{2\pi i}{h}}$ , with $a_j \in \mathbb{Z}$. Then
 \begin{enumerate}[(i)]

 \item If $h=p^n$,  where $p$ is a  prime. Then, up to a rotation, the only irreducible vanishing sums of roots of unity are $1+\zeta_p+...+\zeta_p^{p-1}=0$,  where $\zeta_p$ is  the root of unity of order $p$.  
 \item If $h=p^nq^m$, where $p,q$ are primes. Then, up to a rotation, the only irreducible vanishing sums of roots of unity are $1+\zeta_p+...+\zeta_p^{p-1}=0$ and  $1+\zeta_q+...+\zeta_q^{q-1}=0$, where $\zeta_p$ and $\zeta_q$ are the roots of unity of order $p$ and $q$ respectively.  
 \item If $h=p_1^{n_1}p_2^{n_2}...p_t^{n_t}$, where $p_1,..,p_t$ are different primes, then any $\mathbb{Z}$-linear relation among the $h$-th roots of unity can be obtained from the basic relations $1+\zeta_{p_i}+...+\zeta_{p_i}^{p_{i}-1}=0$, $1 \leq i \leq t$, by addition, substraction and rotation.
  \end{enumerate}
 \end{thm}

\subsection{Automata and generating  function of their language}\label{subsec_automat}
We refer the reader to \cite[p.96]{sims}, \cite[p.7]{epstein}, \cite{pin,pin2}, \cite{epstein-zwik}.
A \emph{finite state automaton} is a quintuple $(S,\Sigma,\mu,Y,s_0)$, where $S$ is a finite set, called the \emph{state set}, $\Sigma$ is a finite set, called the \emph{alphabet}, $\mu:S\times \Sigma \rightarrow S$ is a function, called the \emph{transition function}, $Y$ is a (possibly empty) subset of $S$ called the \emph{accept (or end) states}, and $s_0$ is called the \emph{start state}.  It is a directed  graph with vertices the states and each transition $s \xrightarrow{a} s'$ between states $s$ and $s'$ is an edge with label $a \in \Sigma$. The \emph{label of a path $p$} of length $n$  is the product $a_1a_2..a_n$ of the labels of the edges of $p$.
The  finite state automaton $M=(S,\Sigma,\mu,Y,s_0)$ is \emph{deterministic} if there is only one initial state and each state is the source of exactly one arrow with any given label from  $\Sigma$. In a deterministic automaton, a path is determined by its starting point and its label \cite[p.105]{sims}. It is \emph{co-deterministic} if there is only one final state and each state is the target of exactly one arrow with any given label from  $\Sigma$. The  automaton $M=(S,\Sigma,\mu,Y,s_0)$ is \emph{bi-deterministic} if it is both deterministic and co-deterministic. An automaton $M$ is \emph{complete} if for each state $s\in S$ and for each $a \in \Sigma$, there is exactly one edge from $s$ labelled $a$.
\begin{defn}
Let $M=(S,\Sigma,\mu,Y,s_0)$ be  a finite state automaton. Let $\Sigma^*$ be the free monoid generated by $\Sigma$. Let  $\operatorname{Map}(S,S)$ be  the monoid consisting of all maps from $S$ to $S$. The map $\phi: \Sigma \rightarrow \operatorname{Map}(S,S) $ given by $\mu$ can be extended in a unique way to a monoid homomorphism $\phi: \Sigma^* \rightarrow \operatorname{Map}(S,S)$. The range of this map is a monoid called \emph{the transition monoid of $M$}, which is generated by $\{\phi(a)\mid a\in \Sigma\}$. An element $w \in \Sigma^*$ is \emph{accepted} by $M$ if the corresponding element of $\operatorname{Map}(S,S)$, $\phi(w)$,  takes $s_0$ to an element of the accept states set $Y$. The set $ L\subseteq \Sigma^*$  recognized by $M$ is called \emph{the language accepted by $M$}, denoted by $L(M)$.
\end{defn}
For any directed  graph with $d$ vertices or any finite state  automaton $M$, with alphabet $\Sigma$, and $d$ states,  there exists a  square matrix $A$ of order $d\times d$, with $a_{ij}$ equal to the number of directed edges from  vertex $i$ to vertex $j$, $1\leq i,j\leq d$. This matrix is   non-negative (i.e $a_{ij}\geq 0$) and it is  called  \emph{the transition matrix} (as in \cite{epstein-zwik}) or  \emph{the adjacency matrix} (as in \cite[p.575]{stanley}). For any $k \geq 1$, $(A^k)_{ij}$ is  equal to the number of directed paths of length $k$   from  vertex $i$ to vertex $j$. The function  $p_{i  j}(z)=\sum\limits_{k=0}^{k=\infty}(A^k)_{ij}\,\,z^k\,$ is called \emph{the generating function of $M$} \cite[p.574]{stanley}.  As $\sum\limits_{k=0}^{k=\infty}A^kz^k=(I-zA)^{-1}$, the generating function $p_{ij}(z)$ is equal to  $((I-zA)^{-1})_{ij}$ and it satisfies:
\begin{thm}\cite[p.574]{stanley}\label{theo_genfn_stanley}
The generating function $p_{ij}(z)$ is given by
\[p_{ij}(z)=\frac{(-1)^{i+j}det(I-zA:j,i)}{det(I-zA)}\]
where $(B : j, i)$ denotes the matrix obtained by removing the $j$th row and $i$th column of B, $det(I-zA)$ is the reciprocal polynomial of the characteristic polynomial of $A$. In particular, $p_{ij}(z)$ is a rational function.
\end{thm}

Note that if   $M$ is a bi-deterministic automaton  with alphabet $\Sigma$, $d$ states, start state $i$, accept state  $j$  and  transition matrix $A$, then   $(A^k)_{ij}$ is  the number of  words of length $k$ in the free monoid $\Sigma^*$  accepted by $M$.
\subsection{Irreducible non-negative matrices}
We refer to \cite[Ch.16]{bellman}, \cite[Ch.8]{meier}, \cite[p.536-551]{karlin}. There is a vast literature on the topic. Let $A$ be a transition matrix  of order $d\times d$ of a directed graph or an automaton with $d$ states, as defined in Section 2.2. The matrix $A$ is a non-negative matrix, that is $a_{ij}\geq 0$ for every  $1 \leq i,j \leq d$.  For any $k \geq 1$, $(A^k)_{ij}$ is  equal to the number of directed paths of length $k$   from  vertex $i$ to vertex $j$. If for every $1\leq i,j\leq d$, there exists $m_{ij} \in \mathbb{Z}^+$ such that $(A^{m_{ij}})_{ij}>0$, the matrix is \emph{irreducible}. This condition is equivalent to the graph being strongly-connected, that is any two vertices are connected by a directed path. For $A$  an irreducible non-negative matrix, \emph{the period of $A$} is  the gcd of all $m \in \mathbb{Z}^+$ such that $(A^m)_{ii} >0$ (for any $i$). If the period is $1$, A is called \emph{aperiodic}. In \cite{meier}, an   irreducible and  aperiodic matrix $A$ is  called \emph{primitive} and the period $h$ is called the \emph{index of imprimitivity}.\\

Let A be an irreducible non-negative  matrix of order $d\times d$ with period $h\geq 1$ and spectral radius $r$. Then the Perron-Frobenius theorem states that $r$ is a positive real number and it is a simple eigenvalue of $A$,  $\lambda_{PF}$,  called the \emph{Perron-Frobenius eigenvalue}. It satisfies $\sum\limits_{i}a_{ij}\leq \lambda_{PF} \leq \sum\limits_{j}a_{ij}$. 
The matrix $A$ has a right eigenvector $v_R$ with eigenvalue $\lambda_{PF}$ whose components are all positive and likewise, a 
left eigenvector $v_L$ with eigenvalue $\lambda_{PF}$ whose components are all positive.  Both right and left eigenspaces associated with $\lambda_{PF}$ are one-dimensional.
 \begin{thm}\cite[Thm. V7]{flajolet}\label{theo_genfn_flajolet}
  Let $A$ be a non-negative and irreducible matrix of order $d\times d$. Let $P(z)=(I-zA)^{-1}$. Then all the entries $p_{ij}(z)$ of $P(z)$ (as given in Theorem \ref{theo_genfn_stanley}) have the same radius of convergence $\frac{1}{\lambda_{PF}}$, where $\lambda_{PF}$ is the Perron-Frobenius eigenvalue of $A$.
  \end{thm}
 \begin{rem}\label{rem_periodic}
 The behaviour of irreducible non-negative matrices depends strongly on whether the matrix is aperiodic or not. 
 If $h>1$, then  $A$ has exactly $h$  complex simple eigenvalues with maximal absolute value:  $\lambda_{PF}\,\omega^j$, $\,\,0 \leq j \leq h-1$, where $\omega=e^{\frac{2\pi i}{h}}$ is the root of unity of order $h$. Moreover, the matrix $A$ is similar to $\omega A$, that is  the spectrum of $A$ is invariant under multiplication by $\omega$. Furthermore, the limit $\lim\limits_{k \rightarrow\infty}\frac{1}{k}\sum\limits_{m=0}^{m=k-1}\frac{A^m}{\lambda_{PF}^m}$ exists and is equal to the  $d \times d$ matrix $Q=v_R\,v_L$, with  $v_R$  of order $d \times 1$ and  $v_L$  of order $1 \times d$ such that  $v_L\,v_R=1$. If the matrix $A$ is an irreducible and aperiodic non-negative matrix, then it satisfies many properties similar  to those of the positive matrices.
  \end{rem}
  
\section{Preliminary results}
\subsection{The Schreier automaton of a coset of a subgroup of $F_n$}

We now introduce the particular automata we are interested in, that is \emph{the Schreier coset diagram  for $F_n$ relative to the subgroup  $H$} \cite[p.107]{stilwell} or  \emph{the Schreier automaton for $F_n$ relative to the subgroup $H$} \cite[p.102]{sims}. We refer to \cite{chou_hs} for concrete examples.
\begin{defn}\label{def_Schreier-graph}
 Let $F_n=\langle\Sigma\rangle$ and  $\Sigma^*$ the free monoid generated by $\Sigma$.  Let $H<F_n$ of index $d$. Let $(\tilde{X}_H,p)$ be the covering  of the $n$-leaves bouquet with   vertices  $\tilde{x}_1, \tilde{x}_2,...,\tilde{x}_{d}$ and basepoint $\tilde{x}_i$ for a chosen $ 1\leq i \leq d$. Let  $t_j \in \Sigma^*$ denote the label of a path from $\tilde{x}_i$ to $\tilde{x}_j$. Let $\mathscr{T}=\{1, t_j\mid 1 \leq j\leq d\}$.  Let $\tilde{X}_H$ be the Schreier coset diagram  for $F_n$ relative to the subgroup  $H$,  with $\tilde{x}_i$ representing the subgroup $H$ and the other vertices representing the cosets  $Ht_j$ accordingly.  We call $\tilde{X}_{H}$  \emph{the  Schreier graph  of $H$},  with this correspondence between the vertices   and the cosets  $Ht_j$ accordingly.  
   \end{defn} 
  From its definition, $\tilde{X}_{H}$  is  a strongly-connected graph with $d$ vertices,  so its transition  matrix $A$ is a non-negative and irreducible matrix of order $d \times d$. As $\tilde{X}_{H}$  is a directed $n$-regular graph, the sum of the elements at each row and at each column of  $A$ is equal to $n$. So, from the Perron-Frobenius result for non-negative irreducible matrices, $n$ is the \emph{Perron-Frobenius eigenvalue} of $A$, that is the positive real eigenvalue with maximal absolute value. If $A$ has period $h\geq 1$, then $\{n e^{\frac{2\pi ik}{h}}\mid 0 \leq k\leq h-1\}$ is a set of simple eigenvalues of $A$. Moreover,  $A$ is similar to the matrix $Ae^{\frac{2\pi i}{h}}$, that is the set $\{\lambda e^{\frac{2\pi ik}{h}}\mid 0 \leq k\leq h-1\}$ is a set of eigenvalues of $A$, for each  eigenvalue $\lambda$ of $A$. 
\begin{defn}\label{def_Schreier-automaton}
 Let $F_n=\langle\Sigma\rangle$ and $\Sigma^*$ the free monoid generated by $\Sigma$.  Let $H<F_n$ of index $d$. Let  $\tilde{X}_{H}$  be the  Schreier graph  of $H$. Using the notation from Defn. \ref{def_Schreier-graph}, let $\tilde{x}_i$ be the start state and $\tilde{x}_j$ be the end state for some $1 \leq j \leq d$. We call  the automaton obtained  \emph{the  Schreier automaton   of $Ht_j$}   and denote it by  $\tilde{X}_{Ht_j}$. The language accepted by $\tilde{X}_{Ht_j}$ is  the set of elements in $\Sigma^*$  that belong to $Ht_j$. We call the elements in  $\Sigma^*\cap Ht_j$,  \emph{the positive words in $Ht_j$}. The identity may belong to  this set.
   \end{defn} 
\begin{lem}\label{lem_matrixP}
	Let $H < F_n$ of index $d$, with Schreier graph $\tilde{X}_H$ and transition matrix $A$ with period $h\geq 1$.  Then the following   properties hold:
	\begin{enumerate}[(i)]
		\item $\lambda_{PF}=n$
		\item The vector  $v_R=(1,1,...,1)^T$ of order $1 \times d$ is a right eigenvector of $\lambda_{PF}$ whose components are all positive.
		\item The vector $v_L=\frac{1}{d}(1,1,...,1)$  of order $d \times 1$ is a  left eigenvector of $\lambda_{PF}$ whose components are all positive and such that $v_Lv_R=1$.
		\item  The  matrix $Q=v_R\,v_L$, with  $v_R$  of order $d \times 1$ and  $v_L$  of order $1 \times d$ such that  $v_L\,v_R=1$, is   of order $d\times d$ with all entries equal $\frac{1}{d}$.
		\item If $h=1$, then  $\lim\limits_{k \rightarrow\infty}\frac{A^k}{n^k}=Q$.
		\item If $h>1$, then  $\lim\limits_{k \rightarrow\infty}\frac{1}{k+1}\sum\limits_{m=0}^{m=k}\frac{A^m}{n^m}=Q$.
	\end{enumerate}
\end{lem}
\begin{proof}
	$(i)$, $(ii)$, $(iii)$, $(iv)$  As the sum of every row and every column in $A$ is equal to $n$, $\lambda_{PF}=n$ with right eigenvalue  $v_R=(1,1,...,1)^T$ and left eigenvalue $(1,1,...,1)$. Since $(1,1,...,1)v_R=d$,  $v_L=\frac{1}{d}(1,1,...,1)$  is a left eigenvector that satisfies $v_Lv_R=1$. Computing  $v_R\,v_L$ gives    the matrix $Q$  of order $d\times d$ with all entries equal $\frac{1}{d}$.\\
\noindent	$(v)$, $(vi)$   If $h=1$,  then  $\lim\limits_{k \rightarrow\infty}\frac{A^k}{n^k}=Q$. If $h>1$, $\lim\limits_{k\rightarrow\infty}\frac{1}{k}\sum\limits_{m=0}^{m=k-1}\frac{A^m}{n^m}=Q$ \cite[Ch.8]{meier}.
\end{proof}
\begin{ex}\label{ex-schreier}
	Let $F_2=\langle a,b \rangle$,  $\Sigma=\{a,b\}$.  Let $H= \langle b,a^2,ab^2a,aba^2ba,(ab)^3a \rangle$ be a subgroup of $F_2$ of index $4$, with the following  Schreier graph:
	\begin{figure}[H] 
		\centering \scalebox{0.9}[0.8]{\begin{tikzpicture}
			\SetGraphUnit{4}
			\tikzset{VertexStyle/.append  style={fill}}
			\Vertex[L=$H$, x=-3,y=0]{A}
			\Vertex[L=$Ha$, x=0, y=0]{B}
			
			\Vertex[L=$Hab$, x=3, y=0]{C}
			\Vertex[L=$Haba$, x=6, y=0]{D}
			\Edge[label = a, labelstyle = above](A)(B)
			\Edge[label =a, labelstyle = below](B)(A)
			\Edge[label = b, labelstyle = above](B)(C)
			\Edge[label = b, labelstyle = below](C)(B)
			\Edge[label = b, labelstyle = above](B)(C)
			\Edge[label = a, labelstyle = above](C)(D)
			\Edge[label = a, labelstyle = below](D)(C)
			\Loop[dist = 2cm, dir = NO, label = b, labelstyle = left](A.west)
			\Loop[dist = 2cm, dir = SO, label = b, labelstyle = right](D.east)
			\end{tikzpicture}}
		\caption{The Schreier graph $\tilde{X}_H$ of $H= \langle b,a^2,ab^2a,aba^2ba,(ab)^3a \rangle$ }\label{fig-4sheeted-aperiodic-not_invertible}
	\end{figure}
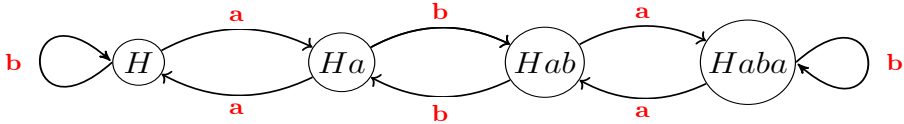

	Let $N=\langle a^2,b^2,ab^2a,aba^2ba,abab\rangle$ be a subgroup of $F_2$ of index $4$, with the following  Schreier graph:
		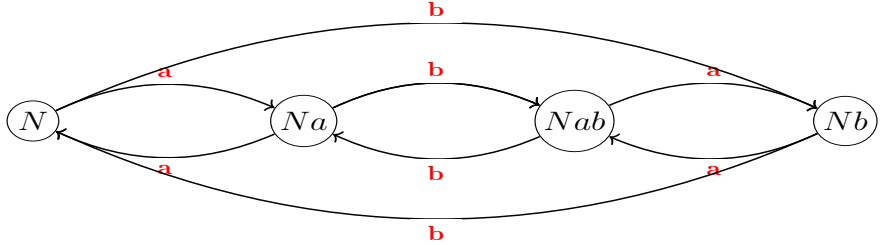
\begin{figure}[H] 
		\centering \scalebox{0.9}[0.7]{\begin{tikzpicture}
			\SetGraphUnit{4}
			\tikzset{VertexStyle/.append  style={fill}}
			\Vertex[L=$N$, x=-6,y=0]{A}
			\Vertex[L=$Na$, x=-2, y=0]{B}
			\Vertex[L=$Nab$, x=2, y=0]{C}
			\Vertex[L=$Nb$, x=6, y=0]{D}
			\Edge[label = a, labelstyle = above](A)(B)
			\Edge[label =a, labelstyle = below](B)(A)
			\Edge[label = b, labelstyle = above](B)(C)
			\Edge[label = b, labelstyle = below](C)(B)
			\Edge[label = b, labelstyle = above](B)(C)
			\Edge[label = a, labelstyle = above](C)(D)
			\Edge[label = a, labelstyle = below](D)(C)
			
			\Edge[label = b, labelstyle = above](A)(D)
			\Edge[label = b, labelstyle = below](D)(A)
			\end{tikzpicture}}\\
		
		\caption{The Schreier graph $\tilde{X}_N$ of $N=\langle a^2,b^2,ab^2a,aba^2ba,abab\rangle$ }\label{fig-4sheeted-period2-not-invertible}
	\end{figure}
	
 The matrices 	$A_H=$  	$ \left( \begin{array}{cccc}
1 & 1 & 0 & 0 \\
1 & 0 & 1 & 0\\
0 &  1 & 0 & 1\\
0 & 0  & 1 & 1
\end{array} \right)$ and $A_N=$	$ \left( \begin{array}{cccc}
0 & 1& 0 & 1 \\
1 & 0 & 1 & 0\\
0 &  1 & 0 & 1\\
1 & 0  & 1 & 0
\end{array} \right)$ are their  respective transition matrices. Both matrices have Perron-Frobenius eigenvalue   $2$ with eigenvectors $v_R=(1,1,1,1)^t$ and  $v_L=\frac{1}{4}(1,1,1,1)$.  Since the period of $A_N$ is $2$,  $A_N$ has simple eigenvalues: $2$,  $-2$,  and  eigenvalue $0$ with multiplicity $2$.  The matrix  $A_H$ is aperiodic.
\end{ex}

\subsection{The generating function of a Schreier automaton}
We prove some properties of the generating function of a Schreier automaton. In the following proposition, we recall several facts proved in \cite{chou-davenport}.
\begin{prop}\cite{chou-davenport}\label{prop-basic-H}
 Let $H < F_n$ of index $d$, with Schreier graph $\tilde{X}_H$ and transition matrix $A$ with period $h\geq 1$. Let $p_{ij}(z)$ denote the generating function of the Schreier automaton, with  $i$ and $j$  the start and end states respectively. Then 
 \begin{enumerate}[(i)]
 \item for every $1 \leq i,j \leq d$, $\frac{1}{n}$ is the radius of convergence of $p_{ij}(z)$. 
\item for every $1 \leq i,j \leq d$, $\{\frac{1}{n} e^{\frac{2\pi ik}{h}}\mid 0 \leq k\leq h-1\}$ is a set of simple poles of $p_{ij}(z)$ of minimal absolute value. 
\item  for $\abs{z}<\frac{1}{n}$, and every $1 \leq i \leq d$, 
$\sum\limits_{j=1}^{j=d}p_{ij}(z)= \frac{1}{1-nz}$.
 \end{enumerate}
\end{prop}

 \begin{defn}
For $1 \leq i,j\leq d$, we define  $m_{ij}$,  $0 \leq m_{ij} \leq d$, to  be the  minimal natural number  such that $(A^{ m_{ij}})_{ij} \neq 0$.
 \end{defn} 
By definition,  if $i \neq j$, then $m_{ij}$ is  the  minimal length of a directed  path from   $i$ to $j$ in $\tilde{X}_H$ and if $i=j$, then $m_{ij}=0$. Note that if $H$ is a subgroup of $\mathbb{Z}=\langle 1 \rangle$ of index $d$, its Schreier graph is a directed loop of length $d$ with each edge labelled by $1$ and its transition matrix $A$ is a permutation matrix with  period $d$ and $m_{ij}=\mathcal{r}$, where $d\mathbb{Z}+\mathcal{r}$ is the coset with  $i$ and $j$  the start and end states respectively. \\

Whenever $h>1$,  only for the exponents $m_{ij}+rh$, $r\geq 0$,  $(A^{ m_{ij}+rh})_{ij} \neq 0$, that is  only positive words of length $m_{ij}+rh$ are accepted by the Schreier automaton, with  $i$ and $j$  the start and end states respectively.   The proportion of  words of  length $m_{ij}+rh$,  for very large $r$, accepted by the automaton is computed in \cite{chou-ijfcs}:  \begin{lem}\cite{chou-ijfcs}
	\label{lem-result-ijfcs}
	Let $H < F_n$ be of index $d$, with Schreier graph $\tilde{X}_H$ and transition matrix $A$ with period $h> 1$. Then,  the following   properties hold:
	\begin{enumerate}[(i)]
		\item $(A^{r})_{ij}=0$,  whenever 
		$r \not\equiv  m_{ij}(mod\,h)$,  $1 \leq i,j \leq d$.
		\item  $\lim\limits_{r\rightarrow\infty}\frac{(A^{m_{ij}+rh})_{ij}}{n^{m_{ij}+rh}}=\frac{h}{d}$,   $1 \leq i,j \leq d$.
		\item  $h$ divides $d$. 
			\end{enumerate}
	\end{lem}
\begin{ex}\label{ex-mij}
	
	Consider  $\tilde{X}_N$ as described in Figure \ref{fig-4sheeted-period2-not-invertible}. If we consider the Schreier  automaton $\tilde{X}_{Na}$ that accepts the positive words in  the coset $Na$, then $m_{12}=1$.   For $\tilde{X}_{Nab}$,  $m_{13}=2$ and for $\tilde{X}_{Nb}$,  $m_{14}=1$.  Clearly, $m_{11}=0$.
	As, $2,-2,0,0$ are the eigenvalues of $A_N$, $(1-2z), (1+2z),1,1$ are the eigenvalues  of  $I-zA$ and $det(I-zA)=(1-4z^2)$. The generating functions are 
 $p_{11}(z)=\frac{1-2z^2}{1-4z^2}$, 
	 $p_{12}(z)=p_{14}(z)=\frac{z}{1-4z^2}$,  and  $p_{13}(z)=\frac{2z^2}{1-4z^2}$. 
\end{ex}
In the following lemma, we compute the residue of the generating function at the simple pole $\frac{1}{n}$ and show it is a function of $n$ and $d$ only.
\begin{lem}\label{lem_res-of-1n}
	Let $H < F_n$ of index $d$, with Schreier graph $\tilde{X}_H$ and transition matrix $A$ with period $h\geq 1$. Let $p_{ij}(z)$ denote the generating function of the Schreier automaton, with  $i$ and $j$  the start and end states respectively. Then,  
	for every $1 \leq i,j \leq d$, $Res(p_{ij}(z), \frac{1}{n})= -(\frac{1}{nd})$.
\end{lem}
\begin{proof}
	Let $1 \leq i \leq d$. Clearly, 
	$\lim\limits_{z \rightarrow \frac{1}{n}}\, \frac{p_{ij}(z)}{p_{ik}(z)}$ is a finite number, for any $1 \leq j,k 
	\leq d$,  since $\frac{1}{n}$ is a simple pole of both.
Furthermore, 	  $p_{ij}(z)$ and $p_{ik}(z)$  are asymptotically equivalent,  that is $\lim\limits_{z \rightarrow \frac{1}{n}}\, \frac{p_{ij}(z)}{p_{ik}(z)}=1$, for any $1 \leq j,k \leq d$. Indeed, as $z \rightarrow\frac{1}{n}$, the generating functions  $p_{ij}(z)\rightarrow\infty$ and $p_{ik}(z)\rightarrow\infty$ at the same rate, since $p_{ij}(z)$ and $p_{ik}(z)$  are the generating functions of different cosets of the same subgroup. We prove that formally. If  $h>1$, then
from Lemma \ref{lem-result-ijfcs}$(i)$,  $p_{ij}(z)=	\sum\limits_{r=0}^{r=\infty}\,(A^{m_{ij}+rh})_{ij}\,z^{m_{ij}+rh}$ and  $p_{ik}(z)=	\sum\limits_{r=0}^{r=\infty}\,(A^{m_{ik}+rh})_{ik}\,z^{m_{ik}+rh}$,  for  $z \in D$,  where $D=\{z\in \mathbb{C}\mid  \abs{z}< \frac{1}{n}\}$.  From Lemma \ref{lem-result-ijfcs}$(ii)$, 	$\lim\limits_{r\rightarrow\infty}\frac{(A^{m_{ij}+rh})_{ij}}{n^{m_{ij}+rh}}=\frac{h}{d}$, that is
		$\lim\limits_{z\rightarrow \frac{1}{n}}\lim\limits_{r\rightarrow\infty}(A^{m_{ij}+rh})_{ij}\,z^{m_{ij}+rh}\,=\,\lim\limits_{r\rightarrow\infty}\lim\limits_{z\rightarrow \frac{1}{n}}(A^{m_{ij}+rh})_{ij}\,z^{m_{ij}+rh}=\frac{h}{d}$.  Also, $\lim\limits_{z\rightarrow \frac{1}{n}}\lim\limits_{r\rightarrow\infty}(A^{m_{ik}+rh})_{ik}\,z^{m_{ik}+rh}=\frac{h}{d}$. So,  $\lim\limits_{z \rightarrow \frac{1}{n}}\,\frac{p_{ij}(z)}{ p_{ik}(z)}=1$.
		If  $h=1$,  then   $p_{ij}(z)=	\sum\limits_{r=0}^{r=\infty}\,(A^{r})_{ij}\,z^{r}$ and  $p_{ik}(z)=	\sum\limits_{r=0}^{r=\infty}\,(A^{r})_{ik}\,z^{r}$. 
			From Lemma \ref{lem_matrixP}$(v)$,    $\lim\limits_{z \rightarrow \frac{1}{n}}\lim\limits_{r \rightarrow\infty}(A^r)_{ij}\,z^r\,=\,\lim\limits_{z \rightarrow \frac{1}{n}}\lim\limits_{r \rightarrow\infty}(A^r)_{ik}\,z^r\,=\, \frac{1}{d}$,  for every $1 \leq i,j,k \leq d$. So,  $\lim\limits_{z \rightarrow \frac{1}{n}}\,\frac{p_{ij}(z)}{ p_{ik}(z)}=1$.
	As $\lim\limits_{z \rightarrow \frac{1}{n}}\, \frac{p_{ij}(z)}{p_{ik}(z)}=\frac{Res(p_{ij}(z), \frac{1}{n})}{Res(p_{ik}(z), \frac{1}{n})}$,  $Res(p_{ij}(z), \frac{1}{n})=Res(p_{ik}(z), \frac{1}{n})$.  		From Prop. \ref{prop-basic-H}$(iii)$, 
	 $\sum\limits_{j=1}^{j=d}Res(p_{ij}(z),\frac{1}{n})   =Res(\frac{1}{1-nz}, \frac{1}{n})$, so $\sum\limits_{j=1}^{j=d}Res(p_{ij}(z),\frac{1}{n}) = -(\frac{1}{n})$, that is  $Res(p_{ij}(z), \frac{1}{n})=-\frac{1}{n}(\frac{1}{d})$,  for all $1 \leq j \leq d$. 
\end{proof}
\subsection{The residue of the generating function at  special poles}
In this subsection, we consider $H < F_n$ of index $d$, with Schreier graph $\tilde{X}_H$ and transition matrix $A$ with period $h>1$. Let $p_{ij}(z)$ denote the generating function of the Schreier automaton, with  $i$ and $j$  the start and end states respectively.  In the following lemmas, we compute the residue of $p_{ij}(z)$  at the  $h$ poles of the form $\frac{1}{n}\omega^k$, where  $\omega=e^{\frac{2\pi i}{h}}$, the root of unity of order $h$. It is done in  several steps in the following technical lemmas.
 \begin{lem}\label{lem_value-m-d-p}
 		Let $H < F_n$ of index $d$, with Schreier graph $\tilde{X}_H$ and transition matrix $A$ with period $h>1$. Let $p_{ij}(z)=\frac{M(z)}{D(z)}$  denote the generating function of the Schreier automaton, with  $i$ and $j$  the start and end states respectively, where $M(z)$  is the numerator  and    $D(z)=det(I-zA)$  the denominator.  
 Let $q$ be any rational number, $\abs{q} <\frac{1}{n}$. Then
  \begin{enumerate}[(i)]
  \item  $p_{ij}(q\omega)\,=\,\omega^{m_{ij}}\,\,p_{ij}(q)$.
   \item  $p_{ij}(q\omega^\ell)\,=\,(\omega^{\ell})^{ m_{ij}}\,\,p_{ij}(q)$.
  \item $D(q\omega^\ell)=D(q)$, for every $1 \leq \ell \leq h-1$.
  \item $M(q\omega)=\omega^{m_{ij}}\,\ M(q)$.
   \item $m_{ij}$ is the multiplicity of $0$ as a root of $M(z)$ and as a zero of $p_{ij}(z)$.
  \end{enumerate} 
     
    \end{lem}
 \begin{proof}
      To shorten, we write $m$ instead of $m_{ij}$.\\
 $(i)$, $(ii)$  For  $\abs{z}<\frac{1}{n}$,    $p_{ij}(z)=\sum\limits_{r=0}^{r=\infty}  (A^{m+hr})_{ij}\,z^{m+hr}$. So,   $p_{ij}(q\omega)=\,\sum\limits_{r=0}^{r=\infty}(A^{m+hr})_{ij}\,q^{m+hr}\omega^{m+hr}=\,\omega^{m}\sum\limits_{r=0}^{r=\infty}(A^{m+hr})_{ij}\,q^{m+hr}=\omega^{m}p_{ij}(q)$, since $\omega^{h}=1$. That is, $p_{ij}(q\omega)=\omega^{m_{ij}}p_{ij}(q)$ and with the same proof: $p_{ij}(q\omega^\ell)\,=\,(\omega^{\ell})^{ m_{ij}}\,\,p_{ij}(q)$.\\
 $(iii)$ Since $h>1$, for any eigenvalue $\lambda\neq 0$ of $A$ with algebraic multiplicity $n_\lambda$,  $\lambda\omega^j$ is an eigenvalue of $A$ and  $1-z\lambda\omega^j $ is an eigenvalue of   $I-zA$ (with same  multiplicity  $n_\lambda$), for every $0 \leq j\leq h-1$. When $0$ is an eigenvalue of $A$, the corresponding eigenvalue of   $I-zA$ is $1$. So,  $D(z)=det(I-zA)\,=\,\prod\limits_{\lambda}\prod\limits_{r=0}^{r=h-1} (1-z\lambda\omega^r)$;  $D(q)=\,\prod\limits_{\lambda}\prod\limits_{r=0}^{r=h-1} (1-q\lambda\omega^r)$ and $D(q\omega^\ell)=\,\prod\limits_{\lambda}\prod\limits_{r=0}^{r=h-1} (1-q\lambda\omega^{\ell+r})$. As $\omega^{h}=1$,  $D(q\omega^\ell)=D(q)$  for every $1 \leq \ell \leq h-1$. \\
 $(iv)$ By definition, $p_{ij}(z)=\frac{M(z)}{D(z)}$, so $M(q\omega)=p_{ij}(q\omega)D(q\omega)=$
 
 $\omega^{m_{ij}}\,p_{ij}(q)D(q)$ from $(i)$ and $(ii)$, that is $M(q\omega)=\omega^{m_{ij}}M(q)$.\\
 $(v)$ From $(i)$, $p_{ij}(z)=\,\sum\limits_{r=0}^{r=\infty}(A^{m+hr})_{ij}\,z^{m+hr}= \,z^m\sum\limits_{r=0}^{r=\infty}(A^{m+hr})_{ij}\,z^{hr}$, that is $m$ the multiplicity of $0$ as a zero of $p_{ij}(z)$ and as a root of $M(z)$.
  
\end{proof}

 From Lemma \ref{lem_value-m-d-p}$(ii)$,   we recover the fact that if  $\frac{1}{n}$ is  a simple pole of $p_{uj}(z)$ then $\frac{1}{n}e^{\frac{2\pi i\ell}{h}}$ is also a simple pole of $p_{ij}(z)$, for every $0 \leq \ell\leq h-1$. Indeed, Lemma \ref{lem_value-m-d-p}$(ii)$ implies 
   $\lim\limits_{z \rightarrow\frac{1}{n}\omega^\ell}\, p_{ij}(z)=\omega^{\ell m_{ij}}\lim\limits_{z \rightarrow \frac{1}{n}}\, p_{ij}(z)$.

    \begin{lem}\label{lem_res-dz}
  Let $p_{ij}(z)=\frac{M(z)}{D(z)}$. Then 
  $Res(\frac{1}{D(z)}, \frac{1}{\lambda} \omega^\ell)=\, \omega^\ell\,  Res(\frac{1}{D(z)}, \frac{1}{\lambda})$,  for every    non-zero eigenvalue $\lambda$  of $A$ and every $1 \leq \ell\leq h-1$. In particular,  $Res(\frac{1}{D(z)}, \frac{1}{n} \omega^\ell)=\, \omega^\ell\,  Res(\frac{1}{D(z)}, \frac{1}{n})$.
    \end{lem}
    \begin{proof}
 As in the proof of Lemma \ref{lem_value-m-d-p},  $D(z)\,=\,\prod\limits_{\rho}\prod\limits_{r=0}^{r=h-1} (1-z\rho\omega^r)$. So, $\frac{1}{\rho}\omega^r$ is a pole of $\frac{1}{D(z)}$, for every    eigenvalue $\rho\neq 0$  of $A$ and every $0 \leq r\leq h-1$. Consider  $\lambda \neq 0$ a particular eigenvalue of $A$ and let $d(z)=\,\prod\limits_{\rho\neq \lambda}\prod\limits_{r=0}^{r=h-1} (1-z\rho\omega^r)$, so $D(z)=\,\prod\limits_{r=0}^{r=h-1} (1-z\lambda\omega^r)d(z)$.
 It holds that $d(\frac{1}{\lambda})=\,\prod\limits_{\rho\neq \lambda}\prod\limits_{r=0}^{r=h-1} (1-\frac{\rho}{\lambda}\,\omega^r)$ and   $d(\frac{1}{\lambda}\omega^\ell)=\,\prod\limits_{\rho\neq \lambda}\prod\limits_{r=0}^{r=h-1} (1-\frac{\rho}{\lambda}\,\omega^{\ell+r})$. That is,  $d(\frac{1}{\lambda}\omega^\ell)\,=\,d(\frac{1}{\lambda})$, since $w^h=1$. Now, $ Res(\frac{1}{D(z)}, \frac{1}{\lambda})= \lim \limits_{z\rightarrow\frac{1}{\lambda}}(z-\frac{1}{\lambda})\frac{1}{D(z)}\,=\, \lim \limits_{z\rightarrow\frac{1}{\lambda}}(-\frac{1}{\lambda})(1-\lambda z)\frac{1}{D(z)}\,=\,(-\frac{1}{\lambda})\,\frac{1}{d(\frac{1}{\lambda})}\;\;\frac{1}{\prod\limits_{r=1}^{r=h-1} (1-\omega^r)}$
and  $ Res(\frac{1}{D(z)}, \frac{1}{\lambda}\,\omega^\ell)= \lim \limits_{z\rightarrow\frac{1}{\lambda}\omega^\ell}(z-\frac{1}{\lambda}\omega^\ell)\frac{1}{D(z)}\,=\, \lim \limits_{z\rightarrow\frac{1}{\lambda}\omega^\ell}(-\frac{1}{\lambda}\omega^\ell)(1-\lambda\omega^{h-\ell} z)\frac{1}{D(z)}\,=\,\omega^\ell\,(-\frac{1}{\lambda})\,\frac{1}{d(\frac{1}{\lambda}\omega^\ell)}\;\;\frac{1}{\prod\limits_{r\in I} (1-\omega^{\ell+r})}$, where $I= \{ 0\leq r\leq h-1\mid r\neq h-\ell \}$. Since $d(\frac{1}{\lambda}\omega^\ell)\,=\,d(\frac{1}{\lambda})$, it remains to show that $\prod\limits_{r\in I} (1-\omega^{\ell+r})=
      \prod\limits_{r=1}^{r=h-1} (1-\omega^r)$. 
As $w^h=1$,  $\prod\limits_{r\in I} (1-\omega^{\ell+r})= \prod\limits_{r=1}^{r=h-1} (1-\omega^r)$. So, 
 $Res(\frac{1}{D(z)}, \frac{1}{\lambda} \omega^\ell)=\, \omega^\ell\,  Res(\frac{1}{D(z)}, \frac{1}{\lambda})$,  for every    eigenvalue $\lambda \neq 0$  of $A$ and every $1 \leq \ell\leq h-1$ and in particular, for the Perron-Frobenius eigenvalue $n$,  $Res(\frac{1}{D(z)}, \frac{1}{n} \omega^\ell)=\, \omega^\ell\,  Res(\frac{1}{D(z)}, \frac{1}{n})$.
  \end{proof}

   \begin{lem}\label{lem_mz}
 For every  $1 \leq i,j \leq d$,  $p_{ij}(z)=\frac{M(z)}{D(z)}$ satisfies:
      
    \begin{enumerate}[(i)]
    \item   $Res(p_{ij}(z), \frac{1}{n} \omega)= -(\frac{1}{nd})\omega^{m_{ij}+1}$.
    \item  $Res(p_{ij}(z), \frac{1}{n} \omega^\ell)= -(\frac{1}{nd})(\omega^{\ell})^{m_{ij}+1}$.
    \end{enumerate}
      \end{lem}
        \begin{proof}
$(i)$ $ Res(p_{ij}(z) ,\frac{1}{n})=\lim \limits_{z\rightarrow\frac{1}{n}}(z-\frac{1}{n})\frac{M(z)}{D(z)}\,=\,Res(\frac{1}{D(z)}, \frac{1}{n})\,M(\frac{1}{n})$ and 
   $ Res(p_{ij}(z) ,\frac{1}{n}\omega)=\lim \limits_{z\rightarrow\frac{1}{n}\omega}(z-\frac{1}{n}\omega)\frac{M(z)}{D(z)}\,=\,Res(\frac{1}{D(z)}, \frac{1}{n}\omega)\,M(\frac{1}{n}\omega)$. From Lemma \ref{lem_value-m-d-p}$(iv)$,   $M(\frac{1}{n}\omega)=\omega^{m_{ij}}\,M(\frac{1}{n})$  and  from Lemma \ref{lem_res-dz}, $Res(\frac{1}{D(z)}, \frac{1}{n}\omega)=\omega\,Res(\frac{1}{D(z)}, \frac{1}{n})$, so  $ Res(p_{ij}(z), \frac{1}{n}\omega)=\,
\omega^{m_{ij}+1} \, Res(p_{ij}(z), \frac{1}{n})=\, -(\frac{1}{nd})\,\omega^{m_{ij}+1}$, with the last equality  from Lemma \ref{lem_res-of-1n}.\\
 $(ii)$ As in $(i)$,  $ Res(p_{ij}(z), \frac{1}{n}\omega^\ell)=\,
  \omega^\ell Res(\frac{1}{D(z)}, \frac{1}{n})\,(\omega^\ell)^{m_{ij}}M(\frac{1}{n}) =$\\
  $(\omega^\ell)^{m_{ij}+1}  Res(p_{ij}(z), \frac{1}{n})\,=\, -(\frac{1}{nd})(\omega^{\ell})^{m_{ij}+1}$.
  \end{proof}
\begin{ex}\label{ex-residues}
	
	Consider  $\tilde{X}_N$ as described in Figure \ref{fig-4sheeted-period2-not-invertible}. The generating functions are 
	$p_{11}(z)=\frac{1-2z^2}{1-4z^2}$, 
	$p_{12}(z)=p_{14}(z)=\frac{z}{1-4z^2}$,  and  $p_{13}(z)=\frac{2z^2}{1-4z^2}$. 
	For all $1\leq  j \leq 4 $,  $Res(p_{1j}(z), \frac{1}{2})=-\frac{1}{8}$; 
	   $ Res(p_{11}(z),- \frac{1}{2})=\,Res(p_{13}(z),- \frac{1}{2})=\,\frac{1}{8}$ and 
	$ Res(p_{12}(z),- \frac{1}{2})=\,  Res(p_{14}(z),- \frac{1}{2})=\,-\frac{1}{8}$.

\end{ex}
 \section{The HS conjecture as a problem in vanishing sum of roots of unity and convex polygons} 
  Let $F_n$ be the free group on $n \geq 2$ generators. Let $\{H_i\alpha_i\}_{i=1}^{i=s}$ be a coset  partition with $H_i<F_n$ of index $d_i$, $\alpha_i \in F_n$, $1 \leq i \leq s$, and $1<d_1 \leq ...\leq d_s$.  Let $\tilde{X}_{i}$ denote the  Schreier  graph of $H_i$, with  transition matrix   $A_i$ of period $h_i\geq 1$, $1 \leq i\leq s$.  Let  $\tilde{X}_{H_i\alpha_i}$ denote the  Schreier  automaton  of $H_i\alpha_i$, with generating function $p_i(z)$, $1 \leq i\leq s$. Let denote by $m_i$ the minimal natural number such that $(A_i^{m_i})_{k_il_i}\neq 0$, where $k_i$ and $l_i$ denote the initial and final states corresponding to $\tilde{X}_{H_i\alpha_i}$. In this section, we assume $h=max\{h_i \mid 1 \leq i \leq s\}\,>\,1$. We denote  $\omega=e^{\frac{2\pi i}{h}}$ and   $J=\{j \,\mid\, h_j=h, \,1 \leq j\leq s\}$.  Under these conditions, we prove in \cite{chou-davenport}:
    \begin{prop}\cite{chou-davenport} \label{prop_sum_res=0}  Assume $h>1$, where  $h=max\{h_i \mid 1 \leq i \leq s\}$.  Let $J=\{j \,\mid\, 1 \leq j\leq s,\,h_j=h\}$. Then
    there  is a repetition of the maximal period, that is $\mid J \mid \geq 2$. 
     \end{prop}
 \subsection{Proof of Theorem $1$}
  We show that Proposition \ref{prop_sum_res=0} induces a set of  irreducible vanishing sums of roots of unity of order $h$  that describe convex polygons, which is the content of Theorem $1$.  The following proposition is a slightly more extended version of Theorem $1$.
  
   \begin{prop}\label{prop_vanish}
    Assume $h>1$, where  $h=max\{h_i \mid 1 \leq i \leq s\}$.  Let $J=\{j \,\mid\, 1 \leq j\leq s,\,h_j=h\}$.  The  coset partition   $P=\{H_i\alpha_i\}_{i=1}^{i=s}$, $\alpha_i \in F_n$  induces a
vanishing sum of roots of unity of order $h$  
\[\sum\limits_{j\in J}(\prod\limits_{\substack{i\in J \\ i\neq j}}d_i)\,(\omega^{})^{m_{j}}\,=\,0\]
      Furthermore, $P$ induces 
    
\begin{enumerate}[(i)]
	\item    a set of 
	irreducible  vanishing sums of roots of unity of order $h$  of  thefollowing  form, where $J'\subseteq J$ :
	\[ \sum\limits_{j\in J'}(\prod\limits_{\substack{i\in J' \\ i\neq j}}d_i)\,(\omega^{})^{m_{j}}\,=\,0\]    
	\item a set of convex polygons, where each polygon   $\mathcal{P}$ has  $\mid J'\mid$ sides.
	\item each $J'$ satisfies $\mid J'\mid \geq p$, where $p$ is the smallest prime dividing $h$.
\end{enumerate}    
    \end{prop}

\begin{proof}
	From Proposition \ref{prop-basic-H}$(iii)$, it results that 
	$\sum\limits_{j\in J}Res(p_{j}(z), \frac{1}{n}w^{\ell})=0$,  for every $\ell$ with $gcd(\ell,h)=1$.  In particular, $\sum\limits_{j\in J}Res(p_{j}(z), \frac{1}{n}w)=0$.  From Lemma \ref{lem_mz}, $\sum\limits_{j\in J}Res(p_{j}(z), \frac{1}{n}w)= -\sum\limits_{j\in J}(\frac{1}{nd_j})(\omega)^{m_{j}+1}$, that is,  after reduction,  we have  $\sum\limits_{j\in J}(\frac{1}{d_j})\,(\omega^{})^{m_{j}}\,=\,0$ and after multiplication by $ \prod\limits_{\substack{j\in J}}d_j$, we have    
	$ \sum\limits_{j\in J}(\prod\limits_{\substack{i\in J \\ i\neq j}}d_i)\,(\omega^{})^{m_{j}}\,=\,0$, a vanishing sum of roots of unity of order $h$ with  positive  integer coefficients. 	\\
$(i)$, $(ii)$ 	If this  vanishing sum of roots of unity of order $h$  is not irreducible, then it admits several proper irreducible  non-empty subsums that vanish. That is,  there exists an irreducible vanishing sum of roots of unity of order $h$,  $\sum\limits_{j\in J'}(\prod\limits_{\substack{i\in J' \\ i\neq j}}d_i)\,(\omega^{})^{m_{j}}\,=\,0$,  for each $J'\subseteq J$,  (obtained   after simplification by  $\prod\limits_{\substack{i\in J\setminus J' }}d_i$).   Since all the coefficients in the vanishing sum are positive integers,  each such irreducible vanishing sum  describes a convex polygon $\mathcal{P}$  with  $\mid J'\mid$ sides \cite{mann}. This polygon may be degenerate, as it  may occur that $m_{j_{1}}= m_{j_{2}} $, for some $j_{1},j_{2}\in J'$. \\
	$(iii)$ From \cite[p.92]{lam}, any non-empty vanishing sum of $h$-th roots of unit must have length at least $p$,  where $p$ is the smallest prime dividing $h$, so $\mid J'\mid \geq p$.
	 \end{proof} 
Note that from Proposition \ref{prop_vanish}, the number of repetitions of the maximal period $h$ is at least $p$, where $p$ is the smallest prime dividing $h$. In particular, in the case of $\mathbb{Z}$, the maximal period is $d_s$, so  we recover the following result  for $\mathbb{Z}$:   the largest index $d_s$ appears at least $p$ times,  where $p$ is the smallest prime dividing $d_s$ \cite{newman,znam,sun2}.
\begin{ex}\label{ex-partition}
Let $L= \langle a^2,b,aba \rangle$ be a subgroup  of index $2$ in $F_2$. 
Let $N$ be the subgroup described in Fig. \ref{fig-4sheeted-period2-not-invertible}.  Consider the following coset partition: $\{H_i\alpha_i\}_{i=1}^{i=3}$, with $H_1\alpha_1=L$, $H_2\alpha_2=Na$, $H_3\alpha_3=Nab$, that is    $F_2=L\cup Na \cup Nab$. 
From Proposition \ref{prop_sum_res=0},  there is a repetition of the maximal period $2$, that is $h_2=h_3=2$.  It holds that $h_1=1$, $m_2=1$, $m_3=2$ and the induced  vanishing sum of roots of order 2 is   $d_3\omega +d_2\omega^2=0$ (since $\frac{1}{d_2}\omega^{m_2} +\frac{1}{d_3}\omega^{m_3}=\frac{1}{d_2}\omega +\frac{1}{d_3}\omega^2=0$),   which implies $d_2=d_3$. Clearly, the induced  polygon here  is degenerate. 
\end{ex}
\subsection{Proof of Theorem 2}
We show now that   a coset partition $P$ has multiplicity if and only if one of the polygons obtained from the induced  vanishing sum have  at least two edges of the same length.

\begin{proof}[Proof of Theorem $2$] 
From Proposition \ref {prop_vanish},  each irreducible non-empty subsum 
  $\sum\limits_{j\in J'}(\prod\limits_{\substack{i\in J' \\ i\neq j}}d_i)\,(\omega^{})^{m_{j}}\,=\,0$, with $J'\subseteq J$,  describes a convex polygon  $\mathcal{P}$   with  $\mid J'\mid$ sides. Moreover,  the length of each edge has the form $\prod\limits_{\substack{i\in J' \\ i\neq j}}d_i$, for some $j \in J'$.
So,  $\mathcal{P}$  has at least two edges of the same length if and only if  $\prod\limits_{\substack{i\in J' \\ i\neq j}}d_i\,=\,\prod\limits_{\substack{i\in J' \\ i\neq k}}d_i$, for some $j ,k\in J'$, $j \neq k$, that is  if and only if  $d_j=d_k$.
\end{proof} 
From the proof of Proposition \ref{prop_vanish},  $\omega$ is a root of the polynomial   $g(z)=\sum\limits_{j\in J'}(\prod\limits_{\substack{i\in J' \\ i\neq j}}d_i)\,(z^{})^{m_{j}}$. As the cyclotomic polynomial of order $h$, $\Phi_h(z)$, is the minimal polynomial of $\omega$  in $\mathbb{Q}(z)$ and  $g(z)$  in $\mathbb{Q}(z)$,  $\Phi_h(z)$  divides  $g(z)$  and every primitive root of unity of order $h$ is a root of    $g(z)$. Clearly, this implies $max\{m_j\mid j \in J'\}$ is at least $\varphi(h)$, the Euler function of $h$.  So, we proved the following corollary:
\begin{cor}\label{cor}
	Let    $g(z)=\sum\limits_{j\in J'}(\prod\limits_{\substack{i\in J' \\ i\neq j}}d_i)\,(z^{})^{m_{j}}$ in $\mathbb{Q}(z)$ as defined above. 
	Then, $\Phi_h(z)$, the  cyclotomic polynomial of order $h$,   divides $g(z)$. Moreover, $max\{m_j\mid j \in J'\}$ is at least $\varphi(h)$, the Euler function of $h$.  
\end{cor}   

\subsection{Proof of Theorem 3}
Using the above construction,  and general  results on  irreducible vanishing sum of roots of unity of order $h$ , we prove, under certain conditions on  $h$, that  all the associated polygons are regular and this implies the  coset partition has multiplicity.

\begin{proof}[Proof of Theorem $3$] 
		Let $J=\{j \,\mid\, 1 \leq j\leq s,\,h_j=h\}$. Let $\sum\limits_{j\in J}(\prod\limits_{\substack{i\in J \\ i\neq j}}d_i)\,(\omega^{})^{m_{j}}\,=\,0$      be the  induced  vanishing sum of roots of unity of order $h$ 	 induced by $P$,  and $\sum\limits_{j\in J'}(\prod\limits_{\substack{i\in J' \\ i\neq j}}d_i)\,(\omega^{})^{m_{j}}\,=\,0$, where  $J' \subseteq J$, an   irreducible sub-sum. We consider  its induced polygon $\mathcal{P}$,  and apply Theorem \ref{lam}.
Indeed,  if $h=p^n$,  where $p$ is a  prime,  then, up to a rotation, the only irreducible vanishing sums of roots of unity are $1+\zeta_p+...+\zeta_p^{p-1}=0$,  where $\zeta_p$ is  the root of unity of order $p$. That is,  all the edges of $\mathcal{P}$ have equal length and from Theorem $2$, $P$ has multiplicity.  If $h=p^nq^m$, where $p,q$ are primes, then, up to a rotation, the only irreducible vanishing sums of roots of unity are $1+\zeta_p+...+\zeta_p^{p-1}=0$ and  $1+\zeta_q+...+\zeta_q^{q-1}=0$, where $\zeta_p$ and $\zeta_q$ are the roots of unity of order $p$ and $q$ respectively.  That is,  all the edges of $\mathcal{P}$ have equal length. In both  cases, all the  induced polygons  are  regular.
\end{proof}
In Example \ref{ex-partition}, $h=2$ and the induced  vanishing sum of roots of order 2,   $d_3\omega +d_2\omega^2=0$, is simply 
$1+\zeta_2$ times an integer  number ($d_2=d_3$).

Note that	if $h=p_1^{n_1}p_2^{n_2}...p_t^{n_t}$, where $p_1,p_2,...,p_t$, $t>2$, are different primes, then, from Theorem \ref{lam},  an   irreducible  vanishing sum of roots of unity of order $h$  can be obtained from 	$\sum\limits_{i=1}^{i=t}\,(\sum\limits_{k=1}^{k=c_i}\epsilon_{i,k}\,\xi_{i,k})\,(1+\zeta_{p_i}+\zeta^2_{p_i}+...+\zeta^{p_i-1}_{p_i})\,=\,0$, 
where $\epsilon_{i,k} \in \{ -1,0,1 \}$, $\xi_{i,k}$ is any root of unity,  $\zeta_{p_i}$ is the  root of unity of order $p_i$, $1 \leq i \leq t$,  $1 \leq k \leq c_i$. In this case, we do not know how to prove that at least one of the polygons associated to the irreducible  vanishing sums of roots of unity has at least two sides of  the same length.


\bigskip\bigskip\noindent
{ Fabienne Chouraqui,}

\smallskip\noindent
University of Haifa at Oranim, Israel.
\smallskip\noindent
E-mail: {\tt fabienne.chouraqui@gmail.com} {\tt fchoura@sci.haifa.ac.il}

\end{document}